\documentclass[a4paper,11pt, reqno]{amsart}
\usepackage{amssymb,amsthm,amsmath}
\usepackage{cite}

\usepackage{amsmath,amsthm,amscd,amssymb}
\usepackage{latexsym}
\usepackage[colorlinks,citecolor=red,pagebackref,hypertexnames=false]{hyperref}
 
\numberwithin{equation}{section}

\theoremstyle{plain}
\newtheorem{theorem}{Theorem}[section]
\newtheorem{Def}[theorem]{Definition}

\newtheorem{corollary}[theorem]{Corollary}
\newtheorem{proposition}[theorem]{Proposition}

\theoremstyle{definition}
\newtheorem{definition}[theorem]{Definition}

\theoremstyle{remark}
\newtheorem{remark}[theorem]{Remark}

\newtheorem{case[theorem]}{Case}

\def \R{{\mathbb R}}

\def \C{{\mathbb C}}

\def\H{{\mathbb H}}

\def\norm#1.#2.{\lVert#1\rVert_{#2}}

\def\R{\mathbb R}

\def \M{{\mathcal M}}

\def \W{{\mathcal W}}

\def \H{{\mathcal H}}

\title[Uncertainty principles for the windowed Opdam--Cherednik transform]{Uncertainty principles for the windowed Opdam--Cherednik transform}

\author{Shyam Swarup Mondal} 
\author{Anirudha Poria}
\thanks{Research supported by ERC Starting Grant No. 713927.}

\keywords{ Opdam--Cherednik transform; windowed Opdam--Cherednik transform;  Heisenberg-type uncertainty inequality; Donoho--Stark’s uncertainty principle; Benedicks-type  uncertainty principle}

\subjclass[2010]{Primary 44A15; Secondary 42C20, 33C45.}

\date{\today}
\begin{document}
\maketitle

\address{Shyam Swarup Mondal, Department of Mathematics, Indian Institute of Technology Delhi, New Delhi 110016, India}\\

\address{Anirudha Poria, Department of Mathematics, Bar-Ilan University, Ramat-Gan 5290002, Israel}\\

\address{Department of Mathematics, SRM Institute of Science and Technology, Kattankulathur 603203, Tamil Nadu, India}\\

\thanks{Corresponding Author: Anirudha Poria, Department of Mathematics, Bar-Ilan University, Ramat-Gan 5290002, Israel}\\
\email{anirudhamath@gmail.com, anirudhp@srmist.edu.in}

\begin{abstract} 
In this paper, we study a few versions of the uncertainty principle for the windowed Opdam--Cherednik transform. In particular, we establish the uncertainty principle for orthonormal sequences, Donoho--Stark's uncertainty principle, Benedicks-type uncertainty principle, Heisenberg-type uncertainty principle and local uncertainty inequality for this transform. We also obtain the Heisenberg-type uncertainty inequality using the $k$-entropy of the windowed Opdam--Cherednik transform. 
\end{abstract}

\section{Introduction}
\noindent Uncertainty principles play a fundamental role in the field of  mathematics, physics, in addition to  some engineering areas such as  signal processing, image processing, quantum theory,  optics, and many other well known areas  \cite{ric14, deb15, gro01, dono89, bial85}. In quantum physics, it says  that a particle’s speed and position cannot both be measured with infinite precision.  The classical Heisenberg uncertainty principle was established  in the Schr\"odinger space (the square-integrable function space). It states that a non-zero function and its Fourier transform cannot be both sharply localized, i.e.,  it is impossible for a non-zero function and its Fourier transform to be simultaneously small.  This phenomenon has been under intensive study for almost a century now and extensively investigated in different settings. One can formulate different forms of the uncertainty principle depending on various ways of measuring the localization of a function.    Uncertainty principles can be divided  into  two  categories:  quantitative and qualitative uncertainty principles. Quantitative uncertainty principles are some special types of inequalities that   tells  us  how a function and its Fourier transform relate. For example,  Benedicks \cite{ben95}, Donoho and Stark \cite{dono89}, and  Slepian and Pollak \cite{sle61} gave quantitative uncertainty principles for the Fourier transforms.  On the other hand, qualitative uncertainty principles   imply the vanishing of a function under some strong conditions on the function. In particular,    Hardy \cite{har33}, Morgan \cite{mor34}, Cowling and Price \cite{cow83}, and Beurling \cite{hor1991} theorems are the  examples of qualitative uncertainty principles. For a more detailed study   on the history of the uncertainty principle, and for many other generalizations and variations of the uncertainty principle, we refer   to the book of Havin and J\"oricke \cite{havin},  and  the excellent survey of Folland and Sitaram \cite{fol97}.\\

\noindent Considerable attention has been devoted to discovering generalizations to new contexts for  quantitative uncertainty principles. For example, quantitative uncertainty principles  were studied in \cite{majjoli16} for the generalized Fourier transform.   Over the  years, discovering new mathematical formulations of the uncertainty principle for the windowed Fourier transform have drawn significant attention among many researchers, see for example \cite{bona03,Wilc00,havin} and the references therein.  Uncertainty principles are further    extended and studied  in different settings by many authors. For instance, one can look \cite{gab19} for the Dunkl Gabor transform,  \cite{bra20} for two sided quaternion windowed Fourier transform,  \cite{bac16, ham17}  for the continuous Hankel transform,  \cite{wen20, gab14} for the windowed Hankel transform, \cite{hammami} for the Hankel--Gabor  transform,  and  \cite{jahn16} for the Heckman--Opdam transform.\\
 
\noindent  As a generalization of Euclidean uncertainty principles for the Fourier transform, Daher et al. \cite{daher12} have obtained some qualitative uncertainty principles for the Cherednik transform. Also,  qualitative uncertainty principles for the Opdam--Cherednik transform have been studied by  Mejjaoli in \cite{ majjoli14}  using classical uncertainty principles for the Fourier transform and composition properties of the Opdam–Cherednik transform. These results have been further extented to modulation spaces by the second author in \cite{por21}.  Moreover, the  Benedicks-type uncertainty principle for the Opdam--Cherednik transform has been investigated by Achak and Daher  in \cite{dah18}. Recently, the second author introduced  the  windowed Opdam--Cherednik transform and discussed the time-frequency analysis of localization operators associated  with the windowed Opdam--Cherednik transform in \cite{por2021}. However, up to our knowledge, quantitative  uncertainty principles have not been studied for the windowed Opdam–Cherednik transform. In this paper, we extend  some quantitative uncertainty principles for this  transform. \\

\noindent The motivation to prove these types of quantitative uncertainty principles for  the windowed Opdam--Cherednik transform in the framework  of the Opdam--Cherednik transform arises from the classical uncertainty principles for the windowed Fourier transform and the remarkable contribution of the Opdam--Cherednik transform  in harmonic analysis (see \cite{and15,opd95,opd00,sch08}).  Another important motivation to study the Jacobi--Cherednik operators arises from their relevance in the algebraic description of exactly solvable quantum many-body systems of Calogero--Moser--Sutherland type (see \cite{die00,hik96}) and they provide a useful tool in the study of special functions with root systems (see\cite{dun92,hec91}). These describe algebraically integrable systems in one dimension and have gained considerable interest in mathematical physics. Other motivation for the investigation of the Jacobi--Cherednik operator and the windowed Opdam--Cherednik transform is to generalize the previous subjects which are bound with the physics. For a more detailed discussion, one can  see \cite{mej}.\\

\noindent Uncertainty inequalities are some special class of uncertainty principles that give us information about how a function and its Fourier transform relate and have got considerable importance in signal analysis, physics, optics, and many other well-known areas \cite{bial85, dono89, gro01, hua18, ric14}. A well-known example of uncertainty inequality is the Heisenberg inequality. 
Considerable attention has been devoted to discovering generalizations to new contexts for uncertainty inequalities for various generalized transforms by several researchers. For instance, these uncertainty inequalities were obtained in \cite{ jahn16} for the Heckman--Opdam transform, and in \cite{mon22} for the Opdam--Cherednik transform. Some recent work on the windowed Opdam--Cherednik transform \cite{por2021, mond22} motivates us to study a few quantitative uncertainty principles for this transform. \\

\noindent  The main aim of this paper is to study a few  uncertainty principles related to the windowed Opdam--Cherednik transform. More preciously, we prove the uncertainty principle for orthonormal sequences, Donoho--Stark's uncertainty principle, Benedicks-type uncertainty principle, Heisenberg-type uncertainty principle and local uncertainty inequality for the windowed Opdam--Cherednik transform. We study the version of Donoho--Stark's uncertainty principle and show that the windowed Opdam--Cherednik transform cannot be concentrated in any small set. Also, we obtain an estimate for the size of the essential support of this transform under the condition that the windowed Opdam--Cherednik transform of a non-zero function is time-frequency concentrated on a measurable set. Then, we investigate the Benedicks-type uncertainty principle  and show that the windowed Opdam--Cherednik transform cannot be concentrated inside a set of measures arbitrarily small. Further, we study the Heisenberg-type uncertainty inequality  for a general magnitude and provide the result related to the $L^2(\R,A_{\alpha,\beta})$-mass of the windowed Opdam--Cherednik transform outside sets of finite measure. Finally, we obtain the Heisenberg-type uncertainty inequality   using the $k$-entropy and study the localization of the $k$-entropy of the windowed Opdam--Cherednik transform.\\

\noindent The presentation of this manuscript is divided into four sections apart from the introduction.   In Section \ref{sec2}, we present some preliminaries related to the  Opdam--Cherednik transform. In Section \ref{sec3}, we recall some essential properties related to the windowed Opdam--Cherednik transform.  Some different types of uncertainty principles associated with the windowed Opdam--Cherednik transform are provided in Section \ref{sec4}. In particular,  we prove the uncertainty principle for orthonormal sequences,  Donoho--Stark's uncertainty principle,  Benedicks-type uncertainty principle, Heisenberg-type uncertainty principle and local uncertainty inequality for the windowed Opdam--Cherednik transform. We conclude the paper with the Heisenberg-type uncertainty inequality  using the $k$-entropy of this transform.

\section{Harmonic analysis and the Opdam--Cherednik transform}\label{sec2}
In this section, we recall some necessary definitions and results from the harmonic analysis related to the   Opdam--Cherednik transform, which will be used frequently. A complete account of harmonic analysis related to this transform can be found in  \cite{and15, mej14, opd95, opd00, sch08, joh15, por21, por2021}. However, we will use the notations given in \cite{por21}.

Let $T_{\alpha, \beta}$ denote the Jacobi--Cherednik differential--difference operator (also called the Dunkl--Cherednik operator)
\[T_{\alpha, \beta} f(x)=\frac{d}{dx} f(x)+ \Big[ 
(2\alpha + 1) \coth x + (2\beta + 1) \tanh x \Big] \frac{f(x)-f(-x)}{2} - \rho f(-x), \]
where $\alpha, \beta$ are two parameters satisfying $\alpha \geq \beta \geq -\frac{1}{2}$ and $\alpha > -\frac{1}{2}$, and $\rho= \alpha + \beta + 1$. Let $\lambda \in \C$. The Opdam hypergeometric functions $G^{\alpha, \beta}_\lambda$ on $\R$ are eigenfunctions $T_{\alpha, \beta} G^{\alpha, \beta}_\lambda(x)=i \lambda  G^{\alpha, \beta}_\lambda(x)$ of $T_{\alpha, \beta}$ that are normalized such that $G^{\alpha, \beta}_\lambda(0)=1$. The eigenfunction $G^{\alpha, \beta}_\lambda$ is given by
\[G^{\alpha, \beta}_\lambda (x)= \varphi^{\alpha, \beta}_\lambda (x) - \frac{1}{\rho - i \lambda} \frac{d}{dx}\varphi^{\alpha, \beta}_\lambda (x)=\varphi^{\alpha, \beta}_\lambda (x)+ \frac{\rho+i \lambda}{4(\alpha+1)} \sinh 2x \; \varphi^{\alpha+1, \beta+1}_\lambda (x),  \]
where $\varphi^{\alpha, \beta}_\lambda (x)={}_2F_1 \left(\frac{\rho+i \lambda}{2}, \frac{\rho-i \lambda}{2} ; \alpha+1; -\sinh^2 x \right) $ is the classical Jacobi function.

For every $ \lambda \in \C$ and $x \in  \R$, the eigenfunction
$G^{\alpha, \beta}_\lambda$ satisfy
\[ |G^{\alpha, \beta}_\lambda(x)| \leq C \; e^{-\rho |x|} e^{|\text{Im} (\lambda)| |x|},\] 
where $C$ is a positive constant. Since $\rho > 0$, we have
\begin{equation}\label{eq1}
	|G^{\alpha, \beta}_\lambda(x)| \leq C \; e^{|\text{Im} (\lambda)| |x|}. 
\end{equation}
Let us denote by $C_c (\R)$ the space of continuous functions on $\R$ with compact support. The Opdam--Cherednik transform is the Fourier transform in the trigonometric Dunkl setting, and it is defined as follows.
\begin{Def}
	Let $\alpha \geq \beta \geq -\frac{1}{2}$ with $\alpha > -\frac{1}{2}$. The Opdam--Cherednik transform $\mathcal{H}_{\alpha, \beta} (f)$ of a function $f \in C_c(\R)$ is defined by
	\[ \H_{\alpha, \beta} (f) (\lambda)=\int_{\R} f(x)\; G^{\alpha, \beta}_\lambda(-x)\; A_{\alpha, \beta} (x) dx \quad \text{for all } \lambda \in \C, \] 
	where $A_{\alpha, \beta} (x)= (\sinh |x| )^{2 \alpha+1} (\cosh |x| )^{2 \beta+1}$. The inverse Opdam--Cherednik transform for a suitable function $g$ on $\R$ is given by
	\[ \H_{\alpha, \beta}^{-1} (g) (x)= \int_{\R} g(\lambda)\; G^{\alpha, \beta}_\lambda(x)\; d\sigma_{\alpha, \beta}(\lambda) \quad \text{for all } x \in \R, \]
	where $$d\sigma_{\alpha, \beta}(\lambda)= \left(1- \dfrac{\rho}{i \lambda} \right) \dfrac{d \lambda}{8 \pi |C_{\alpha, \beta}(\lambda)|^2}$$ and 
	$$C_{\alpha, \beta}(\lambda)= \dfrac{2^{\rho - i \lambda} \Gamma(\alpha+1) \Gamma(i \lambda)}{\Gamma \left(\frac{\rho + i \lambda}{2}\right)\; \Gamma\left(\frac{\alpha - \beta+1+i \lambda}{2}\right)}, \quad \lambda \in \C \setminus i \mathbb{N}.$$
\end{Def}

The Plancherel formula is given by 
\begin{equation}\label{eq03}
	\int_{\R} |f(x)|^2 A_{\alpha, \beta}(x) dx=\int_\R \H_{\alpha, \beta} (f)(\lambda) \overline{\H_{\alpha, \beta} ( \check{f})(-\lambda)} \; d \sigma_{\alpha, \beta} (\lambda),
\end{equation}
where $\check{f}(x):=f(-x)$.

Let $L^p(\R,A_{\alpha, \beta} )$ (resp. $L^p(\R, \sigma_{\alpha, \beta} )$), $p \in [1, \infty] $, denote the $L^p$-spaces corresponding to the measure $A_{\alpha, \beta}(x) dx$ (resp. $d | \sigma_{\alpha, \beta} |(x)$). 

The generalized translation operator associated with the Opdam--Cherednik transform is defined by \cite{ank12}
\begin{equation}\label{eq04}
	\tau_x^{(\alpha, \beta)} f(y)=\int_{\R} f(z) \; {d\mu}_{\;x, y}^{(\alpha, \beta)}(z),
\end{equation}
where ${d\mu}_{\;x, y}^{(\alpha, \beta)}$ is given by 
\begin{equation}\label{eq05}
	{d\mu}_{\;x, y}^{(\alpha, \beta)}(z)=
	\begin{cases} 
		\mathcal{K}_{\alpha, \beta}(x,y,z)\; A_{\alpha, \beta}(z)\; dz & \text{if} \;\; xy \neq 0 \\
		d \delta_x (z) & \text{if} \;\; y=0 \\
		d \delta_y (z) & \text{if} \;\; x=0
	\end{cases}  
\end{equation}
and
\begin{equation*}
	\begin{aligned}
		\mathcal{K}_{\alpha, \beta} {} & (x,y,z) 
		=M_{\alpha, \beta} |\sinh x \cdot \sinh y \cdot \sinh z  |^{-2 \alpha} \int_0^\pi g(x, y, z, \chi)_+^{\alpha-\beta-1} 
		\\
		& \times \left[ 1 - \sigma^\chi_{x,y,z}+ \sigma^\chi_{x,z,y} + \sigma^\chi_{z,y,x} + \frac{\rho}{\beta+\frac{1}{2}} \coth x \cdot \coth y \cdot \coth z (\sin \chi)^2 \right] \times (\sin \chi)^{2 \beta}\; d \chi
	\end{aligned}
\end{equation*}
if $x, y, z \in \R \setminus \{0\} $ satisfy the triangular inequality  $||x|-|y||<|z|<|x|+|y|$, and $\mathcal{K}_{\alpha, \beta} (x,y,z) =0$ otherwise. Here 
\[ \sigma^\chi_{x,y,z}=
\begin{cases} 
	\frac{\cosh x \cdot \cosh y - \cosh z \cdot \cos \chi}{\sinh x \cdot \sinh y} & \text{if} \;\; xy \neq 0 \\
	0 & \text{if} \;\; xy = 0 
\end{cases}
\quad \text{for} \; x, y, z \in \R, \; \chi \in [0,\pi], \]
$g(x, y, z, \chi) = 1- \cosh^2 x - \cosh^2 y - \cosh^2 z + 2 \cosh x \cdot \cosh y \cdot \cosh z \cdot \cos \chi$, and 
\[ g_+=
\begin{cases} 
	g &  \text{if} \;\;  g> 0 \\
	0 & \text{if} \;\;  g \leq 0.
\end{cases} \]
The kernel $\mathcal{K}_{\alpha, \beta} (x,y,z)$ satisfies the following symmetry properties:

$\mathcal{K}_{\alpha, \beta} (x,y,z)=\mathcal{K}_{\alpha, \beta} (y,x,z)$, $\mathcal{K}_{\alpha, \beta} (x,y,z)=\mathcal{K}_{\alpha, \beta} (-z,y,-x)$, $\mathcal{K}_{\alpha, \beta} (x,y,z)=\mathcal{K}_{\alpha, \beta} (x,-z,-y)$.

For every $x,y \in \R$, we have
\begin{equation}\label{eq06}
	\tau_x^{(\alpha, \beta)} f(y) = \tau_y^{(\alpha, \beta)} f(x),
\end{equation}
and 
\begin{equation}\label{eq07}
	\H_{\alpha, \beta}(\tau_x^{(\alpha, \beta)} f )(\lambda)=G_\lambda^{\alpha, \beta} (x) \; \H_{\alpha, \beta} (f)(\lambda),
\end{equation}
for $f \in C_c(\R)$.

If $f \in L^1(\R,A_{\alpha, \beta} )$, then
\begin{equation}\label{eq08}
	\int_\R \tau_x^{(\alpha, \beta)} f(y) \; A_{\alpha, \beta}(y) \; dy = \int_\R f(y) \; A_{\alpha, \beta}(y)\; dy. 
\end{equation}
For every $f \in L^p(\R,A_{\alpha, \beta} )$ and every $x \in \R$, the function $\tau_x^{(\alpha, \beta)} f$ belongs to the space $L^p(\R,A_{\alpha, \beta} )$ and 
\begin{equation}\label{eq09}
	\left\Vert \tau_x^{(\alpha, \beta)} f \right\Vert_{L^p(\R,A_{\alpha, \beta} )} \leq C_{\alpha, \beta} \left\Vert f \right\Vert_{L^p(\R,A_{\alpha, \beta} )}, 
\end{equation}
where $C_{\alpha, \beta}$ is a positive constant.

The convolution product associated with the Opdam--Cherednik transform is defined for two suitable functions $f$ and $g$  by \cite{ank12}
\[ (f *_{\alpha, \beta} g) (x)=\int_\R \tau_x^{(\alpha, \beta)} f(-y) \;g(y) \; A_{\alpha, \beta}(y) \; dy  \]
and
\begin{equation}\label{eq10}
	\H_{\alpha, \beta} (f *_{\alpha, \beta} g)= \H_{\alpha, \beta} (f) \; \H_{\alpha, \beta} (g).
\end{equation}

\section{The windowed Opdam--Cherednik transform}\label{sec3}
In this section, we  collect the  necessary definitions and results from the harmonic analysis related to the windowed Opdam–Cherednik transform. For a detailed discussion on this transform, we refer  to \cite{por2021}.

Let $g \in L^2(\R,A_{\alpha, \beta} )$ and $\xi \in \R$, the modulation operator of $g$ associated with the Opdam--Cherednik transform is defined by 
\begin{equation}\label{eq11}
	\M^{(\alpha, \beta)}_\xi g=\H^{-1}_{\alpha, \beta} \left( \sqrt{\tau_\xi^{(\alpha, \beta)} |\H_{\alpha, \beta}(g)|^2 } \right) .
\end{equation}
Then, for every $g \in L^2(\R,A_{\alpha, \beta} )$ and $\xi \in \R$, by using the Plancherel formula (\ref{eq03}) and 
the translation invariance of the Plancherel measure 
$d \sigma_{\alpha, \beta}$,  we obtain
\begin{equation}\label{eq12}
	\left\Vert \M^{(\alpha, \beta)}_\xi g \right\Vert_{L^2(\R,A_{\alpha, \beta} )}=\left\Vert g \right\Vert_{L^2(\R,A_{\alpha, \beta} )}.
\end{equation}
Now, for a non-zero window function $g \in L^2(\R,A_{\alpha, \beta} )$ and $(x, \xi) \in \R^2$, we consider the function $g_{x, \xi}^{(\alpha, \beta)}$ defined by 
\begin{equation}\label{eq13}
	g_{x, \xi}^{(\alpha, \beta)}= \tau_x^{(\alpha, \beta)} \M^{(\alpha, \beta)}_\xi g.
\end{equation}
For any function $f \in L^2(\R,A_{\alpha, \beta} )$, we define the windowed Opdam--Cherednik transform by
\begin{equation}\label{eq14}
	\W^{(\alpha, \beta)}_g(f)(x,\xi)=\int_\R f(s) \; \overline{g_{x, \xi}^{(\alpha, \beta)}(-s)} \;  A_{\alpha, \beta}(s) \; ds, \quad (x,\xi) \in \R^2,
\end{equation}
which can be also written in the form
\begin{equation}\label{eq15}
	\W^{(\alpha, \beta)}_g(f)(x,\xi)= \left(f *_{\alpha,\beta} \overline{ \M^{(\alpha, \beta)}_\xi g} \right)(x).
\end{equation}
We define the measure $A_{\alpha, \beta} \otimes \sigma_{\alpha, \beta}$ on $\R^2$ by 
\begin{equation}\label{eq16}
	d(A_{\alpha, \beta} \otimes \sigma_{\alpha, \beta})(x, \xi)= A_{\alpha, \beta}(x)  dx \; d\sigma_{\alpha, \beta}(\xi).
\end{equation}

The windowed Opdam--Cherednik transform satisfies the following properties. 
\begin{proposition}\cite{por2021}
	Let $g \in L^2(\R,A_{\alpha, \beta} )$ be a non-zero window function.  Then we have
	
	$(1)$ $($Plancherel's formula$)$ For every $f \in L^2(\R,A_{\alpha, \beta})$, 
	\begin{equation}\label{eq17}
	\left\Vert \W^{(\alpha, \beta)}_g(f) \right\Vert_{L^2(\R^2,  A_{\alpha, \beta}\otimes \sigma_{\alpha, \beta})} = \Vert f \Vert_{L^2(\R,A_{\alpha, \beta})} \; \Vert g \Vert_{L^2(\R,A_{\alpha, \beta})}.
	\end{equation}
	
	$(2)$ $($Orthogonality relation$)$  For every $f , h \in L^2(\R,A_{\alpha, \beta})$, we have
	\begin{equation}\label{eq18}
	\iint_{\R^2} \W^{(\alpha, \beta)}_g(f)(x, \xi) \; \overline{ \W^{(\alpha, \beta)}_g(h)(x, \xi)} \; d(A_{\alpha, \beta} \otimes \sigma_{\alpha, \beta})(x, \xi)  = \Vert g \Vert^2_{L^2(\R,A_{\alpha, \beta})} \int_\R f(s) \overline{h(s)} \;A_{\alpha, \beta}(s) \;ds.
	\end{equation}

	$(3)$ $($Reproducing kernel Hilbert space$)$ The space $\W^{(\alpha, \beta)}_g(L^2(\R,A_{\alpha, \beta}))$ is a reproducing kernel Hilbert space in $L^2(\R^2,A_{\alpha, \beta} \otimes \sigma_{\alpha, \beta} )$ with kernel function $K_g$ defined by 
	\begin{eqnarray}\label{eq19}
	K_g ((x', \xi');(x, \xi))
	& = & \frac{1}{\Vert g \Vert^2_{L^2(\R,A_{\alpha, \beta})}} \; \left( g_{x, \xi}^{(\alpha, \beta)} (- \; \cdot) *_{\alpha, \beta} \overline{  \M^{(\alpha, \beta)}_{\xi'} g} \right)(x') \nonumber \\ 
	& = & \frac{1}{\Vert g \Vert^2_{L^2(\R,A_{\alpha, \beta})}} \; \W^{(\alpha, \beta)}_g \left( g_{x, \xi}^{(\alpha, \beta)} (- \; \cdot) \right) (x', \xi').
	\end{eqnarray}
	Furthermore, the kernel is pointwise bounded
	\begin{equation}\label{eq20}
	\left| K_g ((x', \xi');(x, \xi)) \right| \leq 1, \quad \text{for all}\; (x, \xi);   (x', \xi') \in \R^2. 
	\end{equation}
\end{proposition}

The windowed Opdam--Cherednik transform also satisfies the following boundedness properties.
\begin{proposition}\cite{por2021}
	Let $g \in L^2(\R, A_{\alpha, \beta} )$ be a non-zero window function. Then for every $f \in L^2(\R,A_{\alpha, \beta} )$, we have  
	
	$(1)$   
	\begin{equation}\label{eq21}
	\left\Vert \W^{(\alpha, \beta)}_g(f) \right\Vert_{L^\infty(\R^2,A_{\alpha, \beta} \otimes \sigma_{\alpha, \beta})} \leq \Vert f \Vert_{L^2(\R,A_{\alpha, \beta})} \; \Vert g \Vert_{L^2(\R,A_{\alpha, \beta})}.
	\end{equation}
	
	$(2)$ The function $\W^{(\alpha, \beta)}_g(f) \in L^p(\R^2,A_{\alpha, \beta} \otimes \sigma_{\alpha, \beta})$, $p \in [2, \infty)$ and 
	\begin{equation}\label{eq22}
	\left\Vert \W^{(\alpha, \beta)}_g(f) \right\Vert_{L^p(\R^2,A_{\alpha, \beta} \otimes \sigma_{\alpha, \beta})} \leq \Vert f \Vert_{L^2(\R,A_{\alpha, \beta})} \; \Vert g \Vert_{L^2(\R,A_{\alpha, \beta})}.
	\end{equation}  
\end{proposition}

\section{Uncertainty principles for the windowed Opdam--Cherednik transform}\label{sec4}
In this section, we obtain various uncertainty principles for the windowed Opdam--Cherednik transform. 
\subsection{Uncertainty principle for orthonormal sequences}
In this subsection, we establish  the uncertainty principle for orthonormal sequences associated with  the windowed Opdam--Cherednik transform.  First, we consider the following orthogonal projections:
\begin{enumerate}
	\item  Let $P_{g}$ be the orthogonal projection from $L^2(\R^2,A_{\alpha, \beta} \otimes \sigma_{\alpha, \beta})$ onto $\W_{g}^{(\alpha, \beta)}\left(L^{2}\left(\R,A_{\alpha, \beta} \right)\right) $ and   $\operatorname{Im}P_g$  denotes the    range of  $P_{g}$.
	\item Let $P_{\Sigma}$ be the orthogonal projection on $L^2(\R^2,A_{\alpha, \beta} \otimes \sigma_{\alpha, \beta})$ defined by
	\begin{align}\label{eq25}
	P_{\Sigma} F=\chi_{\Sigma} F, \quad F \in L^2(\R^2,A_{\alpha, \beta} \otimes \sigma_{\alpha, \beta}),
	\end{align}
	where $\Sigma \subset \R^2$ and $ \operatorname{Im}P_\Sigma $ is the  range of  $P_{\Sigma}$.
\end{enumerate}
Also, we define $$\left\|P_{\Sigma} P_{g}\right\|=\sup \left\{\left\|P_{\Sigma} P_{g}(F)\right\|_{L^2(\R^2,A_{\alpha, \beta} \otimes \sigma_{\alpha, \beta})}: F \in L^2(\R^2,A_{\alpha, \beta} \otimes \sigma_{\alpha, \beta}), \|F\|_{L^2(\R^2,A_{\alpha, \beta} \otimes \sigma_{\alpha, \beta})}=1\right\}.$$ We first need the following result.
\begin{theorem}\label{eq26}
	Let $g \in L^2(\R,A_{\alpha, \beta})$ be a non-zero window function. Then  for any  $\Sigma  \subset \R^2$ of  finite measure $A_{\alpha, \beta} \otimes \sigma_{\alpha, \beta}(\Sigma)<\infty$,  the operator $P_{\Sigma} P_{g}$ is a Hilbert--Schmidt  operator. Moreover,   we have the following estimation
	$$
	\left\|P_{\Sigma} P_{g}\right\|^{2} \leq A_{\alpha, \beta} \otimes \sigma_{\alpha, \beta}(\Sigma).
	$$
\end{theorem}
\begin{proof}
Since $P_g$ is a projection onto a reproducing karnel Hilbert space, 	for any  function $F \in L^2(\R^2,A_{\alpha, \beta} \otimes \sigma_{\alpha, \beta})$, the orthogonal
	projection $P_{g}$ can be expressed as
	$$
	P_{g}(F)(x, \xi)=\iint_{\R^2} F(x', \xi') K_g ((x', \xi');(x, \xi))\; d(A_{\alpha, \beta} \otimes \sigma_{\alpha, \beta})(x', \xi'),
	$$
	where $K_g ((x', \xi');(x, \xi))$ is  given  by (\ref{eq19}). Using the relation   (\ref{eq25}), we obtain
	$$
	P_{\Sigma} P_{g}(F)(x, \xi)=\iint_{\R^2}\chi_{\Sigma}( x, \xi) F(x', \xi') K_g ((x', \xi');(x, \xi))\; d(A_{\alpha, \beta} \otimes \sigma_{\alpha, \beta})(x', \xi').
	$$
	This shows that the operator $P_{\Sigma} P_{g}$  is an integral operator with  kernel $K((x', \xi');(x, \xi))=\chi_{\Sigma}( x, \xi)   K_g ((x', \xi');(x, \xi))$.  
	Using the relation \eqref{eq19},  Plancherel's formula  (\ref{eq17}),  and Fubini's theorem, we have
	\begin{align*}
	\left\|P_{\Sigma} P_{g}\right\|_{H S}^{2}&=\iint_{\R^2}\iint_{\R^2}\left|K((x', \xi');(x, \xi))\right|^2 \; d(A_{\alpha, \beta} \otimes \sigma_{\alpha, \beta})(x', \xi') \; d(A_{\alpha, \beta} \otimes \sigma_{\alpha, \beta})(x, \xi)
	\\&=\iint_{\R^2}\iint_{\R^2}\left|\chi_{\Sigma}(x, \xi)\right|^{2}\left| K_g ((x', \xi');(x, \xi))\right|^2  d(A_{\alpha, \beta} \otimes \sigma_{\alpha, \beta})(x', \xi')  d(A_{\alpha, \beta} \otimes \sigma_{\alpha, \beta})(x, \xi)\\
	&=  \frac{1}{\Vert g \Vert^4_{L^2(\R,A_{\alpha, \beta})}} 
	\iint_{\Sigma}\left(\iint_{\R^2}  \left|\W^{(\alpha, \beta)}_g \left( g_{x, \xi}^{(\alpha, \beta)} (- \; \cdot) \right) (x', \xi') \right|^2 d(A_{\alpha, \beta} \otimes \sigma_{\alpha, \beta})(x', \xi')  \right)\\&\qquad\qquad\quad\qquad \times d(A_{\alpha, \beta} \otimes \sigma_{\alpha, \beta})(x, \xi)\\
	&\leq \frac{\Vert g \Vert^4_{L^2(\R,A_{\alpha, \beta})}}{\Vert g \Vert^4_{L^2(\R,A_{\alpha, \beta})}} A_{\alpha, \beta} \otimes \sigma_{\alpha, \beta}(\Sigma)=A_{\alpha, \beta} \otimes \sigma_{\alpha, \beta}(\Sigma).
	\end{align*}
	Thus, the operator $P_{\Sigma} P_{g}$ is a   Hilbert--Schmidt operator. Now, the  proof  follows from the fact that $\left\|P_{\Sigma} P_{g}\right\|\leq \left\|P_{\Sigma} P_{g}\right\|_{H S}$.
\end{proof}
In the following,  we obtain the uncertainty principle for orthonormal sequences associated with the windowed Opdam--Cherednik transform.
\begin{theorem}
	Let $g \in L^2(\R,A_{\alpha, \beta})$ be a non-zero window function and  $\left\{ \phi_{n}\right\}_{n \in \mathbb{N}}$ be an orthonormal sequence in $L^2(\R,A_{\alpha, \beta})$. Then for any  $\Sigma  \subset \R^2$ of  finite measure $A_{\alpha, \beta} \otimes \sigma_{\alpha, \beta}(\Sigma)<\infty,$ we have
	$$
	\sum_{n=1}^{N}\left(1-\left\|\chi_{\Sigma^{c}} \W_{g}^{(\alpha, \beta)}\left(  \phi_{n}\right)\right\|_{L^2(\R^2, A_{\alpha, \beta} \otimes \sigma_{\alpha, \beta})}\right) \leq A_{\alpha, \beta} \otimes \sigma_{\alpha, \beta}(\Sigma),
	$$  for every $N\in \mathbb{N}.$
\end{theorem}
\begin{proof}
	Let $\left\{ e_{n}\right\}_{n \in \mathbb{N}}$ be an orthonormal basis for  $L^{2}\left(\R^2, A_{\alpha, \beta} \otimes \sigma_{\alpha, \beta}\right)$. Since $P_{\Sigma} P_{g}$ is  a  Hilbert--Schmidt operator, from Theorem \ref{eq26}, we get 
	$$
	tr \left(P_{g} P_{\Sigma} P_{g}\right)=\sum_{n \in \mathbb{N}}\left\langle P_{g} P_{\Sigma} P_{g} e_{n}, e_{n}\right\rangle_{L^2(\R^2, A_{\alpha, \beta} \otimes \sigma_{\alpha, \beta})}=\left\|P_{\Sigma} P_{g}\right\|_{H S}^{2} \leq A_{\alpha, \beta} \otimes \sigma_{\alpha, \beta}(\Sigma),
	$$
	where $tr \left(P_{g} P_{\Sigma} P_{g}\right)$ denotes the trace of the operator $P_{g} P_{\Sigma} P_{g}$  and 
	$\langle\cdot, \cdot\rangle_{L^2(\R^2,A_{\alpha, \beta} \otimes \sigma_{\alpha, \beta})}$ denotes the  inner product of  $L^2(\R^2,A_{\alpha, \beta} \otimes \sigma_{\alpha, \beta})$.
	Since $\left\{ \phi_{n}\right\}_{n \in \mathbb{N}}$ be an orthonormal sequence in $L^{2}\left(\R,  A_{\alpha, \beta}\right)$,  from the  orthogonality relation  (\ref{eq18}),  we obtain that   $\{ \W_{g}^{(\alpha, \beta)}\left(  \phi_{n}\right)\}_{n \in \mathbb{N}}$ is also an orthonormal sequence in $L^2(\R^2,A_{\alpha, \beta} \otimes \sigma_{\alpha, \beta})$. Therefore
	\begin{align*}
	&\sum_{n=1}^{N}\left\langle P_{\Sigma} \W_{g}^{(\alpha, \beta)}\left(  \phi_{n}\right), \W_{g}^{(\alpha, \beta)}\left(  \phi_{n}\right)\right\rangle_{L^2(\R^2, A_{\alpha, \beta} \otimes \sigma_{\alpha, \beta})} \\
	&=\sum_{n=1}^{N}\left\langle P_{g} P_{\Sigma} P_{g} \W_{g}^{(\alpha, \beta)}\left(  \phi_{n}\right), \W_{g}^{(\alpha, \beta)}\left(  \phi_{n}\right)\right\rangle_{L^2(\R^2, A_{\alpha, \beta} \otimes \sigma_{\alpha, \beta})} \\
	& \leq tr\left(P_{g} P_{\Sigma} P_{g}\right).
	\end{align*}
	Thus, we have
	$$
	\sum_{n=1}^{N}\left\langle P_{\Sigma} \W_{g}^{(\alpha, \beta)}\left(  \phi_{n}\right), \W_{g}^{(\alpha, \beta)}\left(  \phi_{n}\right)\right\rangle_{L^2(\R^2, A_{\alpha, \beta} \otimes \sigma_{\alpha, \beta})}  \leq A_{\alpha, \beta} \otimes \sigma_{\alpha, \beta}(\Sigma).
	$$
	Moreover,  for any $n$ with $1 \leq n \leq N$, using  the Cauchy--Schwarz inequality, we get 
	$$
	\begin{aligned}
&	\left\langle P_{\Sigma} \W_{g}^{(\alpha, \beta)}\left(  \phi_{n}\right), \W_{g}^{(\alpha, \beta)}\left(  \phi_{n}\right)\right\rangle_{L^2(\R^2, A_{\alpha, \beta} \otimes \sigma_{\alpha, \beta})} \\&=1- \left\langle P_{\Sigma^c} \W_{g}^{(\alpha, \beta)}\left(  \phi_{n}\right), \W_{g}^{(\alpha, \beta)}\left(  \phi_{n}\right)\right\rangle_{L^2(\R^2, A_{\alpha, \beta} \otimes \sigma_{\alpha, \beta})} \\
	&\geq 1-\left\|\chi_{\Sigma^c} \W_{g}^{(\alpha, \beta)}\left(  \phi_{n}\right)\right \|_{L^2(\R^2, A_{\alpha, \beta} \otimes \sigma_{\alpha, \beta})}.
	\end{aligned}
	$$
Therefore 
	\begin{align*}
	\sum_{n=1}^{N}\left(1-\left\|\chi_{\Sigma^c} \W_{g}^{(\alpha, \beta)}\left(  \phi_{n}\right)\right \|_{L^2(\R^2, A_{\alpha, \beta} \otimes \sigma_{\alpha, \beta})}  \right) &\leq	\sum_{n=1}^{N}\left\langle P_{\Sigma} \W_{g}^{(\alpha, \beta)}\left(  \phi_{n}\right), \W_{g}^{(\alpha, \beta)}\left(  \phi_{n}\right)\right\rangle_{L^2(\R^2, A_{\alpha, \beta} \otimes \sigma_{\alpha, \beta})} \\& \leq A_{\alpha, \beta} \otimes \sigma_{\alpha, \beta}(\Sigma).
	\end{align*}
	This completes the proof of the theorem.
\end{proof}

\subsection{Donoho--Stark’s uncertainty principle for the windowed Opdam--Cherednik transform}
Here, we study the version of  Donoho--Stark’s uncertainty principle for the windowed Opdam--Cherednik transform. In particular,  we   investigate the case where $f$  and $\W_{g}^{(\alpha, \beta)}(f)$ are close to zero outside measurable sets.   We start with the following result.
\begin{theorem}
	Let $g \in L^2(\R,  A_{\alpha, \beta})$ be a non-zero window function and    $f\in {L^{2}\left(\R ,A_{\alpha, \beta}  \right)} $ such that $f\neq 0$. Then  for any  $\Sigma  \subset \R^2$  and $\varepsilon \geq 0$ such that
	$$
	\iint_{\Sigma}\left|\W_{g}^{(\alpha, \beta)}(f)(x, \xi)\right|^{2} \; d(A_{\alpha, \beta} \otimes \sigma_{\alpha, \beta})(x, \xi) \geq (1-\varepsilon)\|f\|_{L^2(\R,  A_{\alpha, \beta})}^2 \|g\|_{L^2(\R,  A_{\alpha, \beta})}^2,
	$$
	we have
	$$
	A_{\alpha, \beta} \otimes \sigma_{\alpha, \beta}(\Sigma) \geq 1-\varepsilon.
	$$
\end{theorem}
\begin{proof}
	Using the relation    (\ref{eq21}), we get 
	\begin{align*}
	(1-\varepsilon)\|f\|_{L^2(\R,  A_{\alpha, \beta})}^2 \|g\|_{L^2(\R,  A_{\alpha, \beta})}^2 &\leq \iint_{\Sigma}\left|\W_{g}^{(\alpha, \beta)}(f)(x, \xi)\right|^{2} \;d(A_{\alpha, \beta} \otimes \sigma_{\alpha, \beta})(x, \xi)  \\
	&\leq\left\|\W_{g}^{(\alpha, \beta)}(f)\right\|^2_{L^{\infty}\left(\R^2,A_{\alpha, \beta} \otimes \sigma_{\alpha, \beta}\right)} A_{\alpha, \beta} \otimes \sigma_{\alpha, \beta}(\Sigma) \\
	&\leq A_{\alpha, \beta} \otimes \sigma_{\alpha, \beta}(\Sigma)\;\|f\|_{L^2(\R,  A_{\alpha, \beta})}^2 \|g\|_{L^2(\R,  A_{\alpha, \beta})}^2.
	\end{align*}
	Therefore, $
	A_{\alpha, \beta} \otimes \sigma_{\alpha, \beta}(\Sigma) \geq 1-\varepsilon.
	$
\end{proof}

The following proposition shows that the windowed Opdam--Cherednik transform cannot be concentrated in any small set.
\begin{proposition}\label{eq35}
	Let $g \in L^2(\R,  A_{\alpha, \beta})$ be a non-zero window function.
	Then for any function $f \in L^2(\R,  A_{\alpha, \beta})$ and for any     $\Sigma \subset \R^2 $ such that $A_{\alpha, \beta} \otimes \sigma_{\alpha, \beta}(\Sigma)<1$, we have 
	$$
	\left\|\chi_{\Sigma^{c}} \W_{g}^{(\alpha, \beta)}(f)\right\|_{L^{2}\left(\R^2,A_{\alpha, \beta} \otimes \sigma_{\alpha, \beta}\right)} \geq \sqrt{1-A_{\alpha, \beta} \otimes \sigma_{\alpha, \beta}(\Sigma)}\;\|f\|_{L^2(\R,  A_{\alpha, \beta})} \|g\|_{L^2(\R,  A_{\alpha, \beta})} .
	$$
\end{proposition}
\begin{proof}
	For  any  function $f \in L^2(\R,  A_{\alpha, \beta})$, using the relation (\ref{eq21}), we get  
	$$
	\begin{aligned}
	&\left\|\W_{g}^{(\alpha, \beta)}(f)\right\|_{L^{2}\left(\R^2,A_{\alpha, \beta} \otimes \sigma_{\alpha, \beta}\right)}^{2} 
	=\left\|\chi_{\Sigma} \W_{g}^{(\alpha, \beta)}(f)+\chi_{\Sigma^c} \W_{g}^{(\alpha, \beta)}(f)\right\|_{L^{2}\left(\R^2,A_{\alpha, \beta} \otimes \sigma_{\alpha, \beta}\right)}^{2} \\
	& = \left\|\chi_{\Sigma} \W_{g}^{(\alpha, \beta)}(f)\right\|_{L^{2}\left(\R^2,A_{\alpha, \beta} \otimes \sigma_{\alpha, \beta}\right)}^{2}+\left\|\chi_{\Sigma^{c}} \W_{g}^{(\alpha, \beta)}(f)\right\|_{L^{2}\left(\R^2,A_{\alpha, \beta} \otimes \sigma_{\alpha, \beta}\right)}^{2} \\
	& \leq A_{\alpha, \beta} \otimes \sigma_{\alpha, \beta}(\Sigma)\left\|\W_{g}^{(\alpha, \beta)}(f)\right\|_{L^{\infty}\left(\R^2,A_{\alpha, \beta} \otimes \sigma_{\alpha, \beta}\right)}^{2}+\left\|\chi_{\Sigma^{c}} \W_{g}^{(\alpha, \beta)}(f)\right\|_{L^{2}\left(\R^2,A_{\alpha, \beta} \otimes \sigma_{\alpha, \beta}\right)}^{2} \\
	& \leq A_{\alpha, \beta} \otimes \sigma_{\alpha, \beta}(\Sigma)\; \|f\|_{L^2(\R,  A_{\alpha, \beta})}^{2}\; \|g\|_{L^2(\R,  A_{\alpha, \beta})}^{2}+\left\|\chi_{\Sigma^{c}} \W_{g}^{(\alpha, \beta)}(f)\right\|_{L^{2}\left(\R^2,A_{\alpha, \beta} \otimes \sigma_{\alpha, \beta}\right)}^{2}.
	\end{aligned}
	$$
	Thus,  using Plancherel's formula (\ref{eq17}), we obtain 
	$$
	\left\|\chi_{\Sigma^{c}} \W_{g}^{(\alpha, \beta)}(f)\right\|_{L^{2}\left(\R^2,A_{\alpha, \beta} \otimes \sigma_{\alpha, \beta}\right)} \geq \sqrt{1-A_{\alpha, \beta} \otimes \sigma_{\alpha, \beta}(\Sigma)}\;\|f\|_{L^2(\R,  A_{\alpha, \beta})} \|g\|_{L^2(\R,  A_{\alpha, \beta})} .
	$$
\end{proof}

\begin{definition}
	Let $E$ be a measurable subset of $\R$ and $0 \leq \varepsilon_{E}<1$. Then we say that a  function  $f\in L^p(\R,  A_{\alpha, \beta}), 1 \leq p \leq 2,$ is   $\varepsilon_{E}$-concentrated on $E$ in $L^p(\R,  A_{\alpha, \beta})$-norm,   if
	$$
	\left\|\chi_{E^{c}} f\right\|_{L^p(\R,  A_{\alpha, \beta})} \leq \varepsilon_{E}\|f\|_{L^p(\R,  A_{\alpha, \beta})}.
	$$
	If $\varepsilon_{E}=0$, then $E$ contains the support of  $f$ .
\end{definition}

\begin{definition}
	Let $\Sigma$ be a measurable subset of $\R^2$ and $0 \leq \varepsilon_{\Sigma}<1$.   Let $f, g \in L^2(\R,  A_{\alpha, \beta})$ be two non-zero functions. We say that $\W_{g}^{(\alpha, \beta)}(f)$ is $ \varepsilon_{\Sigma}$-time-frequency concentrated on $\Sigma$, if
	$$
	\left\|\chi_{\Sigma^{c}} \W_{g}^{(\alpha, \beta)}(f)\right\|_{L^{2}\left(\R^2,A_{\alpha, \beta} \otimes \sigma_{\alpha, \beta}\right)}  \leq \varepsilon_{\Sigma} \left\|  \W_{g}^{(\alpha, \beta)}(f)\right\|_{L^{2}\left(\R^2,A_{\alpha, \beta} \otimes \sigma_{\alpha, \beta}\right)}.
	$$
\end{definition}
If $\W_{g}^{(\alpha, \beta)}(f)$ is $ \varepsilon_{\Sigma}$-time-frequency concentrated on $\Sigma$, then in the following, we obtain an estimate for the size of the  essential support   of the windowed Opdam--Cherednik transform.
\begin{theorem}
	Let $g \in L^2(\R,  A_{\alpha, \beta})$ be a non-zero window function and
	$f \in L^2(\R,  A_{\alpha, \beta})$   such that $f\neq 0$.  Let    $\Sigma  \subset \R^2$ such that  $A_{\alpha, \beta} \otimes \sigma_{\alpha, \beta}(\Sigma)<\infty$ and $\varepsilon_\Sigma \geq 0$.  If
	$\W_{g}^{(\alpha, \beta)}(f)$ is $ \varepsilon_{\Sigma}$-time-frequency concentrated on $\Sigma,$ then  
	$$A_{\alpha, \beta} \otimes \sigma_{\alpha, \beta}(\Sigma)\geq (1-\varepsilon_\Sigma^2).$$
\end{theorem}
\begin{proof}
Since $\W_{g}^{(\alpha, \beta)}(f)$ is $ \varepsilon_{\Sigma}$-time-frequency concentrated on $\Sigma$, using  Plancherel's formula (\ref{eq17}), we deduce that
	$$
	\begin{aligned}
	\Vert f \Vert_{L^2(\R,  A_{\alpha, \beta})}^2 \; \Vert g \Vert_{L^2(\R,  A_{\alpha, \beta})}^2
	&=	\left\Vert \W^{(\alpha, \beta)}_g(f) \right\Vert_{L^2(\R^2,  A_{\alpha, \beta}\otimes \sigma_{\alpha, \beta})}^2 \\
	&= 	 \left \|   \chi_{\Sigma^{c}}  \W^{(\alpha, \beta)}_g(f) \right\|_{L^2(\R^2,  A_{\alpha, \beta}\otimes \sigma_{\alpha, \beta})}^2 +	 \left \|  \chi_{\Sigma} \W^{(\alpha, \beta)}_g(f) \right\|_{L^2(\R^2,  A_{\alpha, \beta}\otimes \sigma_{\alpha, \beta})}^2 \\
	& \leq\varepsilon_{\Sigma}^2 \left\|  \W_{g}^{(\alpha, \beta)}(f)\right\|_{L^{2}\left(\R^2,A_{\alpha, \beta} \otimes \sigma_{\alpha, \beta}\right)}^2+ \left \|  \chi_{\Sigma}   \W^{(\alpha, \beta)}_g(f) \right\|_{L^2(\R^2,  A_{\alpha, \beta}\otimes \sigma_{\alpha, \beta})}^2.
	\end{aligned}
	$$
	Hence,  using  the relation (\ref{eq21}), we obtain 
	\begin{align}\label{eq30}
	(1-\varepsilon_\Sigma^2) \Vert f \Vert_{L^2(\R,  A_{\alpha, \beta})}^2 \; \Vert g \Vert_{L^2(\R,  A_{\alpha, \beta})}^2 
	&\leq \left \|  \chi_{\Sigma} \W^{(\alpha, \beta)}_g(f) \right\|_{L^2(\R^2,  A_{\alpha, \beta}\otimes \sigma_{\alpha, \beta})}^2\\\nonumber
	& \leq \left\|\W_{g}^{(\alpha, \beta)}(f)\right\|_{L^{\infty}\left(\R^2,A_{\alpha, \beta} \otimes \sigma_{\alpha, \beta}\right)}^{2} \;A_{\alpha, \beta} \otimes \sigma_{\alpha, \beta}(\Sigma)\\ \nonumber
	&\leq  \|f\|_{L^2(\R,  A_{\alpha, \beta})}^{2}\; \|g\|_{L^2(\R,  A_{\alpha, \beta})}^{2}\;A_{\alpha, \beta} \otimes \sigma_{\alpha, \beta}(\Sigma),
	\end{align}
	which completes the proof.
\end{proof}
\begin{theorem}
	Let $\Sigma  \subset \R^2$ such that  $A_{\alpha, \beta} \otimes \sigma_{\alpha, \beta}(\Sigma)<\infty$, $\varepsilon_\Sigma \geq 0$,  $g \in L^2(\R,  A_{\alpha, \beta})$ be a non-zero window function, and
	$f \in L^2(\R,  A_{\alpha, \beta})$ such that $f\neq 0$. If
	$\W_{g}^{(\alpha, \beta)}(f)$ is $ \varepsilon_{\Sigma}$-time-frequency concentrated on $\Sigma$, then for every $p > 2$,  we have 
	$$A_{\alpha, \beta} \otimes \sigma_{\alpha, \beta}(\Sigma)\geq (1-\varepsilon_\Sigma^2)^{\frac{p}{p-2}}.$$
\end{theorem}
\begin{proof}
	Since $\W_{g}^{(\alpha, \beta)}(f)$ is $ \varepsilon_{\Sigma}$-time-frequency concentrated on $\Sigma$, from \eqref{eq30}, we have 
	\begin{align*} 
	(1-\varepsilon_\Sigma^2) \Vert f \Vert_{L^2(\R,  A_{\alpha, \beta})}^2 \; \Vert g \Vert_{L^2(\R,  A_{\alpha, \beta})}^2 
	&\leq \left \|  \chi_{\Sigma} \W^{(\alpha, \beta)}_g(f) \right\|_{L^2(\R^2,  A_{\alpha, \beta}\otimes \sigma_{\alpha, \beta})}^2.
	\end{align*}
	Again, aplying H\"older’s inequality for the conjugate exponents $\frac{p}{2}$ and $\frac{p}{p-2}$, we get
	\begin{align*}
	\left \|  \chi_{\Sigma} \W^{(\alpha, \beta)}_g(f) \right\|_{L^2(\R^2,  A_{\alpha, \beta}\otimes \sigma_{\alpha, \beta})}^2&\leq  \left \|  \W^{(\alpha, \beta)}_g(f) \right\|_{L^p(\R^2,  A_{\alpha, \beta}\otimes \sigma_{\alpha, \beta})}^2(A_{\alpha, \beta} \otimes \sigma_{\alpha, \beta}(\Sigma))^{\frac{p-2}{p}}.
	\end{align*}
	Now, using  the relation \eqref{eq22}, we obtain 
	\begin{align*} 
	\left \|  \chi_{\Sigma} \W^{(\alpha, \beta)}_g(f) \right\|_{L^2(\R^2,  A_{\alpha, \beta}\otimes \sigma_{\alpha, \beta})}^2&\leq\; \|f\|_{L^2(\R, A_{\alpha, \beta})}^{2} \|g\|_{L^2(\R,  A_{\alpha, \beta})}^{2}(A_{\alpha, \beta} \otimes \sigma_{\alpha, \beta}(\Sigma))^{\frac{p-2}{p}}.
	\end{align*}
	Hence, 	$$A_{\alpha, \beta} \otimes \sigma_{\alpha, \beta}(\Sigma)\geq (1-\varepsilon_\Sigma^2)^{\frac{p}{p-2}}.$$
\end{proof}

\begin{theorem}
	Let $\Sigma  \subset \R^2$ such that  $A_{\alpha, \beta} \otimes \sigma_{\alpha, \beta}(\Sigma)<\infty$, $g \in L^2(\R,  A_{\alpha, \beta})$ be a non-zero window function, and
	$f \in L^1(\R, A_{\alpha, \beta}) \cap L^2(\R, A_{\alpha, \beta})$ such that $\|  \W^{(\alpha, \beta)}_g(f) \|_{L^2(\R^2,  A_{\alpha, \beta}\otimes \sigma_{\alpha, \beta})}=1$. Let $E\subset \R$ such that $A_{\alpha, \beta}(E)<\infty$. If $f$ is $ \varepsilon_{E}$-concentrated on $E$ in $ L^1(\R,  A_{\alpha, \beta})$-norm  and 
	$\W_{g}^{(\alpha, \beta)}(f)$ is $ \varepsilon_{\Sigma}$-time-frequency concentrated on $\Sigma$, then  
	$$A_{\alpha, \beta}(E)\geq (1-\varepsilon_E)^2\;\|f\|_{L^1(\R,  A_{\alpha, \beta})}^2\;\|g\|_{L^2(\R,  A_{\alpha, \beta})}^2,$$
	and 
	$$ \|f\|_{L^2(\R, A_{\alpha, \beta})}^{2} \;\|g\|_{L^2(\R,  A_{\alpha, \beta})}^{2}\;A_{\alpha, \beta} \otimes \sigma_{\alpha, \beta}(\Sigma)\geq (1-\varepsilon_\Sigma^2).$$
	In particular,
	$$A_{\alpha, \beta}(E)~ A_{\alpha, \beta} \otimes \sigma_{\alpha, \beta}(\Sigma)\;  \|f\|_{L^2(\R,  A_{\alpha, \beta})}^{2}\geq  (1-\varepsilon_E)^2\;(1-\varepsilon_\Sigma^2) \;\|f\|_{L^1(\R,  A_{\alpha, \beta})}^{2}.$$
	
\end{theorem}
\begin{proof}
	Since $\W_{g}^{(\alpha, \beta)}(f)$ is $ \varepsilon_{\Sigma}$-time-frequency concentrated on $\Sigma$, from \eqref{eq30}, we have 
	\begin{align*}
	(1-\varepsilon_\Sigma^2) \Vert f \Vert_{L^2(\R,  A_{\alpha, \beta})}^2 \; \Vert g \Vert_{L^2(\R,  A_{\alpha, \beta})}^2 
	&\leq \left \|  \chi_{\Sigma} \W^{(\alpha, \beta)}_g(f) \right\|_{L^2(\R^2,  A_{\alpha, \beta}\otimes \sigma_{\alpha, \beta})}^2.
	\end{align*}
	Since $\|  \W^{(\alpha, \beta)}_g(f) \|_{L^2(\R^2,  A_{\alpha, \beta}\otimes \sigma_{\alpha, \beta})}=1,$ using the relation (\ref{eq21}) and   Plancherel's formula (\ref{eq17}),  we obtain
	\begin{align}\label{eq33}\nonumber
	(1-\varepsilon_\Sigma^2)&\leq \left \|  \chi_{\Sigma} \W^{(\alpha, \beta)}_g(f) \right\|_{L^2(\R^2,  A_{\alpha, \beta}\otimes \sigma_{\alpha, \beta})}^2\\\nonumber
	&\leq \left \|   \W^{(\alpha, \beta)}_g(f) \right\|_{L^\infty(\R^2,  A_{\alpha, \beta}\otimes \sigma_{\alpha, \beta})}^2 ~A_{\alpha, \beta} \otimes \sigma_{\alpha, \beta}(\Sigma)\\
	&\leq  \|f\|_{L^2(\R,  A_{\alpha, \beta})}^{2}~\|g\|_{L^2(\R,  A_{\alpha, \beta})}^{2}\;A_{\alpha, \beta} \otimes \sigma_{\alpha, \beta}(\Sigma).
	\end{align}	
Similarly, since $f$ is $ \varepsilon_{E}$-concentrated on $E$ in $ L^1(\R,  A_{\alpha, \beta})$-norm, using the  Cauchy--Schwarz inequality and the fact that $\|f\|_{L^2(\R,  A_{\alpha, \beta})}\|g\|_{L^2(\R,  A_{\alpha, \beta})}=1$,
 we get 
	\begin{align}\label{eq34} 
	(1-\varepsilon_E) \Vert f \Vert_{L^1(\R,  A_{\alpha, \beta})} \leq \left \|  \chi_{E}f \right\|_{L^1(\R,  A_{\alpha, \beta} )}\leq \left \|  f \right\|_{L^2(\R,  A_{\alpha, \beta} )} ~ A_{\alpha, \beta} (E)^{\frac{1}{2}} =\frac{A_{\alpha, \beta} (E)^{\frac{1}{2}}}{\left \|  g\right\|_{L^2(\R,  A_{\alpha, \beta} )} } .
	\end{align}
Finally, from (\ref{eq33}) and (\ref{eq34}), we have 
	$$A_{\alpha, \beta}(E)~ A_{\alpha, \beta} \otimes \sigma_{\alpha, \beta}(\Sigma)\;  \|f\|_{L^2(\R,  A_{\alpha, \beta})}^{2}\geq  (1-\varepsilon_E)^2\;(1-\varepsilon_\Sigma^2) \;\|f\|_{L^1(\R,  A_{\alpha, \beta})}^{2}.$$
	This completes the proof of the theorem.
\end{proof}

\subsection{Benedicks-type uncertainty principle for the windowed Opdam--Cherednik transform}
In this subsection,  we study Benedicks-type uncertainty principle for the windowed Opdam--Cherednik transform. The following proposition shows that this  transform cannot be concentrated inside a set of  measure arbitrarily small. 
\begin{proposition}\label{eq41}
	Let $g \in L^2(\R,  A_{\alpha, \beta})$ 
	be a non-zero window function and $\Sigma\subset \mathbb{R}^2$. If    $\left\|P_{\Sigma} P_{g}\right\|<1$,  then there exists a constant $c(\Sigma, g)>0$ such that for every   $f \in  L^2(\R,  A_{\alpha, \beta}),$ we have
	$$\|f\|_{L^2(\R,  A_{\alpha, \beta})} \|g\|_{L^2(\R,  A_{\alpha, \beta})}\leq  c(\Sigma, g)\left\|\chi_{\Sigma^{c}} \W_{g}^{(\alpha, \beta)}(f)\right\|_{L^{2}\left(\R^2,A_{\alpha, \beta} \otimes \sigma_{\alpha, \beta}\right)}.
	$$
\end{proposition}
\begin{proof}
Since $P_{\Sigma}$ is an   orthogonal projection on $L^2(\R^2,A_{\alpha, \beta} \otimes \sigma_{\alpha, \beta})$,   for any $F\in L^2(\R^2,A_{\alpha, \beta} \otimes \sigma_{\alpha, \beta})$,  we get  
$$
\|P_g(F)\|_{L^{2}\left(\R^2,A_{\alpha, \beta} \otimes \sigma_{\alpha, \beta}\right)}^2=\|P_{\Sigma}P_g(F)\|_{L^{2}\left(\R^2,A_{\alpha, \beta} \otimes \sigma_{\alpha, \beta}\right)}^2+\|P_{\Sigma^c}P_g(F)\|_{L^{2}\left(\R^2,A_{\alpha, \beta} \otimes \sigma_{\alpha, \beta}\right)}^2.
$$
Again, using   the identity $P_{\Sigma}P_g(F)=P_{\Sigma}P_g \cdot P_g(F)$,  we get $$
\|P_{\Sigma}P_g(F)\|_{L^{2}\left(\R^2,A_{\alpha, \beta} \otimes \sigma_{\alpha, \beta}\right)}^2 \leq\|P_{\Sigma}P_g\|^2 \| P_g(F)\|_{L^{2}\left(\R^2,A_{\alpha, \beta} \otimes \sigma_{\alpha, \beta}\right)}^2.
$$
Hence, 
\begin{align}\label{eq44}
\|P_g(F)\|_{L^{2}\left(\R^2,A_{\alpha, \beta} \otimes \sigma_{\alpha, \beta}\right)}^2\leq \frac{1}{1-\|P_{\Sigma}P_g\|^2 } \|P_{\Sigma^c}P_g(F)\|_{L^{2}\left(\R^2,A_{\alpha, \beta} \otimes \sigma_{\alpha, \beta}\right)}^2.
\end{align} 
Since $P_{g}$ is an orthogonal projection from $L^2(\R^2,A_{\alpha, \beta} \otimes \sigma_{\alpha, \beta})$ onto $\W_{g}^{(\alpha, \beta)}\left(L^{2}\left(\R,A_{\alpha, \beta} \right)\right) $,  for any $f\in L^2(\R,  A_{\alpha, \beta})$, using the relation  (\ref{eq44}) and  Plancherel's formula (\ref{eq17}),  we get 
$$\|f\|_{L^2(\R,  A_{\alpha, \beta})} \|g\|_{L^2(\R,  A_{\alpha, \beta})}\leq  \frac{1}{ \sqrt{1-\|P_{\Sigma}P_g\|^2}}\left\|\chi_{\Sigma^{c}} \W_{g}^{(\alpha, \beta)}(f)\right\|_{L^{2}\left(\R^2,A_{\alpha, \beta} \otimes \sigma_{\alpha, \beta}\right)}.$$

The desired result follows  by choosing the constant $c(\Sigma, g)=  \frac{1}{ \sqrt{1-\|P_{\Sigma}P_g\|^2}}$.
\end{proof}
Next,  we obtain Benedicks-type uncertainty principle for  the windowed Opdam--Cherednik transform.
\begin{theorem}\label{eq40}
	Let $g \in L^2(\R,  A_{\alpha, \beta})$ 
	be a non-zero window function  such that  $$\sigma_{\alpha, \beta}\left(  \left\{ \H_{\alpha, \beta} (g)  \neq 0 \right\} \right)<\infty.$$
	Then for any  $\Sigma  \subset \R^2$  such that for almost every $\xi \in  \R, \int_{\R} \chi_{\Sigma}(x, \xi)~A_{\alpha, \beta}(x)\;dx <\infty$, we have
	$$
	\W^{(\alpha, \beta)}_g \left(L^2(\R,  A_{\alpha, \beta})\right)  \cap \operatorname{Im} P_{\Sigma}=\{0\}.
	$$
\end{theorem}
\begin{proof}
	Let   $F\in 	\W^{(\alpha, \beta)}_g \left(L^2(\R,  A_{\alpha, \beta})\right)  \cap  \operatorname{Im} P_{\Sigma}$ be a non-trivial function. Then there exists a function $f\in L^2(\R, A_{\alpha, \beta})$ such that $F=	\W^{(\alpha, \beta)}_g(f)$ and $Supp F \subset \Sigma$.  For any $\xi \in \mathbb{R}$, we consider the function
	$$
	F_{\xi}(x)=\W^{(\alpha, \beta)}_g(f)(x, \xi),\qquad x\in \R.
	$$ Then, we get  $
	Supp F_{\xi} \subset\{x \in \R :  (x, \xi ) \in \Sigma \}.$
	Since  for almost every $\xi \in  \R, \int_{\R} \chi_{\Sigma}(x, \xi) A_{\alpha, \beta} (x)\;dx<\infty$,  we have 
	$
	A_{\alpha, \beta} \left( Supp F_{\xi}\right)<\infty.
	$
	Using the relation (\ref{eq15}), we get
	$$\H_{\alpha, \beta} (F_{\xi})= \H_{\alpha, \beta} (f) \; \H_{\alpha, \beta} (\M^{(\alpha, \beta)}_\xi g) \quad \;a.e.
	$$
	Hence,
	$$Supp \H_{\alpha, \beta} (F_{\xi}) \subset  Supp \tau_\xi^{(\alpha, \beta)} | \H_{\alpha, \beta} (g)|^2,$$
	and from the hypothesis, we get  $\sigma_{\alpha, \beta} \left(\left\{\H_{\alpha, \beta} (F_{\xi}) \neq 0\right\}\right)<\infty$. 
	From    Benedicks-type result for the Opdam--Cherednik transform\cite{dah18},  we    conclude  that, for every  $\xi \in \mathbb{R}, F(\cdot, \xi)=0$, which eventually  implies that $F=0.$
\end{proof}

\begin{remark}
	Let $g \in L^2(\R,  A_{\alpha, \beta})$ 
	be a non-zero window function  such that  $$\sigma_{\alpha, \beta}\left(  \left\{ \H_{\alpha, \beta} (g)  \neq 0 \right\} \right)<\infty.$$ Then, for any non-zero function $f \in L^2(\R,  A_{\alpha, \beta})$, we have $A_{\alpha, \beta} \otimes \sigma_{\alpha, \beta} ( Supp\W^{(\alpha, \beta)}_g(f))=\infty$, i.e., the support of $\W^{(\alpha, \beta)}_g(f)$ cannot be of   finite measure.
\end{remark}

\begin{proposition}
	Let $g \in L^2(\R,  A_{\alpha, \beta})$ 
	be a non-zero window function  such that  $$\sigma_{\alpha, \beta}\left(  \left\{\H_{\alpha, \beta} (g)  \neq 0 \right\} \right)<\infty.$$  Let $\Sigma \subset \R^2$ such that $A_{\alpha, \beta} \otimes \sigma_{\alpha, \beta}(\Sigma)<\infty $,  then there exists a constant $c(\Sigma, g)>0$ such that
	$$
	\|f\|_{L^2(\R,  A_{\alpha, \beta})}\leq c(\Sigma, g ) 	\left\|\chi_{\Sigma^{c}} \W_{g}^{(\alpha, \beta)}(f)\right\|_{L^{2}\left(\R^2,A_{\alpha, \beta} \otimes \sigma_{\alpha, \beta}\right)}.
	$$
\end{proposition}
\begin{proof}
Assume that $$\left\|P_{\Sigma} P_{g}(F)\right\|_{L^{2}\left(\R^2,A_{\alpha, \beta} \otimes \sigma_{\alpha, \beta}\right)} =\|F\|_{L^{2}\left(\R^2,A_{\alpha, \beta} \otimes \sigma_{\alpha, \beta}\right)}, \;F \in L^2(\R^2,A_{\alpha, \beta} \otimes \sigma_{\alpha, \beta}).$$  Since $P_{\Sigma}$
	and $P_{g}$ are  orthogonal projections, we obtain $P_{\Sigma}(F)=P_{g}(F)=F.$ Again, since $A_{\alpha, \beta} \otimes \sigma_{\alpha, \beta}(\Sigma)<\infty $,   for almost every $\xi \in  \R, \int_{\R} \chi_{\Sigma}(x, \xi)\;A_{\alpha, \beta} (x)\;dx<\infty$ and from Theorem \ref{eq40}, we get $F=0.$ Hence,  for $F \neq 0$, we obtain  $$\left\|P_{\Sigma} P_{g}(F)\right\|_{L^{2}\left(\R^2,A_{\alpha, \beta} \otimes \sigma_{\alpha, \beta}\right)} <\|F\|_{L^{2}\left(\R^2,A_{\alpha, \beta} \otimes \sigma_{\alpha, \beta}\right)}. $$  
	Since  $P_{\Sigma} P_{g}$ is a Hilbert--Schmidt operator, we obtain that the   largest eigenvalue $|\lambda|$ of the operator $P_{\Sigma} P_{g}$ satisfy $|\lambda|<1$ and $\left\|P_{\Sigma} P_{g}\right\|=|\lambda|<1$. Finally, using Proposition \ref{eq41},  we get the desired result. 
\end{proof}

\subsection{Heisenberg-type uncertainty principle for the windowed Opdam--Cherednik transform} 
This subsection is devoted to study Heisenberg-type uncertainty inequality for  the windowed Opdam--Cherednik transform for general magnitude $s > 0.$  Indeed, we have the following result.
\begin{theorem}\label{eq39}
	Let $s>0$. Then there exists a constant $c(s, \alpha, \beta)>0$ such that for all $f, g \in L^2(\R,  A_{\alpha, \beta})$, we have
	\begin{align*}
	&\left\|x^s \W_{g}^{(\alpha, \beta)}(f)\right\|_{L^2(\R^2,  A_{\alpha, \beta}\otimes \sigma_{\alpha, \beta})}^2  +\left\|\xi ^s \W_{g}^{(\alpha, \beta)}(f)\right\|_{L^2(\R^2,  A_{\alpha, \beta}\otimes \sigma_{\alpha, \beta})}^2 \\& \geq  c(s, \alpha, \beta )\; \|f\|_{L^2(\R,  A_{\alpha, \beta})}^{2}\;\|g\|_{L^2(\R,  A_{\alpha, \beta})}^{2}.
	\end{align*}
\end{theorem}
\begin{proof}
	Let $r>0$  and $B_{r}=\{(x, \xi) \in \R^2 : ~|(x, \xi)|<r \}$ be the ball with  centre at origin and radius $r$ in 
	$\R^2$.  Fix $\varepsilon_{0} \leq 1$ small enough such that $A_{\alpha, \beta} \otimes \sigma_{\alpha, \beta}(B_{\varepsilon_{0}})<1$. From Proposition \ref{eq35}, we have 
	\begin{align*}
	& \|f\|_{L^2(\R, A_{\alpha, \beta})}^{2}\; \|g\|_{L^2(\R,  A_{\alpha, \beta})}^{2}\\
	& \leq \frac{1}{ \left[1-A_{\alpha, \beta} \otimes \sigma_{\alpha, \beta}(B_{\varepsilon_{0}})\right]} \iint_{|(x, \xi)| \geq \varepsilon_{0}} \left|\W_{g}^{(\alpha, \beta)}(f)(x, \xi)\right|^{2} ~d(A_{\alpha, \beta} \otimes \sigma_{\alpha, \beta})(x, \xi) \\
	& \leq \frac{1}{\varepsilon_{0}^{2 s}\left[1-A_{\alpha, \beta} \otimes \sigma_{\alpha, \beta}(B_{\varepsilon_{0}})\right]} \iint_{|(x, \xi)| \geq \varepsilon_{0}}|(x, \xi)|^{2 s}\left|\W_{g}^{(\alpha, \beta)}(f)(x, \xi)\right|^{2} ~d(A_{\alpha, \beta} \otimes \sigma_{\alpha, \beta})(x, \xi) \\
	& \leq \frac{1}{\varepsilon_{0}^{2 s}\left[1-A_{\alpha, \beta} \otimes \sigma_{\alpha, \beta}(B_{\varepsilon_{0}})\right]} \iint_{\R^2}|(x, \xi)|^{2 s}\left|\W_{g}^{(\alpha, \beta)}(f)(x, \xi)\right|^{2} ~d(A_{\alpha, \beta} \otimes \sigma_{\alpha, \beta})(x, \xi).
	\end{align*}
	Consequentely,
	\begin{align}\label{eq36}
	\left\||(x, \xi)|^{s}\W_{g}^{(\alpha, \beta)}(f)\right\|_{L^2(\R^2,  A_{\alpha, \beta}\otimes \sigma_{\alpha, \beta})}^2  \geq \varepsilon_{0}^{2 s}\left[1-A_{\alpha, \beta} \otimes \sigma_{\alpha, \beta}(B_{\varepsilon_{0}})\right]\; \|f\|_{L^2(\R,  A_{\alpha, \beta})}^{2}\;\|g\|_{L^2(\R,  A_{\alpha, \beta})}^{2}.
	\end{align}
	Finally, using  the fact that $|a+b|^{s} \leq 2^{s}\left(|a|^{s}+|b|^{s}\right)$ and  from (\ref{eq36}), we get 
	\begin{align*}
	&\varepsilon_{0}^{2 s}\left[1-A_{\alpha, \beta} \otimes \sigma_{\alpha, \beta}(B_{\varepsilon_{0}})\right]\; \|f\|_{L^2(\R,  A_{\alpha, \beta})}^{2}\;\|g\|_{L^2(\R,  A_{\alpha, \beta})}^{2}\\
	&\leq \left\||(x, \xi)|^{s}\W_{g}^{(\alpha, \beta)}(f)\right\|_{L^2(\R^2,  A_{\alpha, \beta}\otimes \sigma_{\alpha, \beta})}^2\\
	&\leq 2^s \left\|x^s \W_{g}^{(\alpha, \beta)}(f)\right\|_{L^2(\R^2,  A_{\alpha, \beta}\otimes \sigma_{\alpha, \beta})}^2  +2^s\left\|\xi ^s \W_{g}^{(\alpha, \beta)}(f)\right\|_{L^2(\R^2,  A_{\alpha, \beta}\otimes \sigma_{\alpha, \beta})}^2 .
	\end{align*}
Hence, we get the desired result with  $c(s, \alpha, \beta)=\varepsilon_{0}^{2 s}2^{-s}\left[1-A_{\alpha, \beta} \otimes \sigma_{\alpha, \beta}(B_{\varepsilon_{0}})\right].$ 
\end{proof}
\subsection{Local uncertainty inequality for the windowed Opdam--Cherednik transform}
Here, we   discuss  results related to  the $L^2(\R, A_{\alpha, \beta})$-mass of the windowed Opdam--Cherednik transform outside sets of finite measure. Indeed, we establish the following result.
\begin{theorem} \label{eq37}
	Let $s>0 $. Then there exists a constant $c(s, \alpha, \beta)>0$ such that for any 	$f, g \in L^2(\R,  A_{\alpha, \beta})$  and any   $\Sigma  \subset \R^2$  such that $A_{\alpha, \beta} \otimes \sigma_{\alpha, \beta}(\Sigma)<\infty$, we have
	$$
	\left\|\W_{g}^{(\alpha, \beta)}(f)\right\|_{L^{2}\left(\Sigma,  A_{\alpha, \beta} \right)} \leq c(s, \alpha, \beta)\left[A_{\alpha, \beta} \otimes \sigma_{\alpha, \beta}(\Sigma)\right]^{\frac{1}{2}}\left\||(x, \xi)|^{s}\W_{g}^{(\alpha, \beta)}(f)\right\|_{L^2(\R^2,  A_{\alpha, \beta}\otimes \sigma_{\alpha, \beta})}.
	$$
\end{theorem}
\begin{proof}
	Since $ \|\W_{g}^{(\alpha, \beta)}(f) \|_{L^{2}\left(\Sigma,A_{\alpha, \beta} \otimes \sigma_{\alpha, \beta}\right)}\leq  \|\W_{g}^{(\alpha, \beta)}(f) \|_{L^{\infty}\left(\R^2,A_{\alpha, \beta} \otimes \sigma_{\alpha, \beta}\right)}~ [A_{\alpha, \beta} \otimes \sigma_{\alpha, \beta}(\Sigma)]^{\frac{1}{2}},
	$
	using the relation (\ref{eq21}),  we get 
	$$\left\|\W_{g}^{(\alpha, \beta)}(f)\right\|_{L^{2}\left(\Sigma,  A_{\alpha, \beta} \right)} \leq \left[A_{\alpha, \beta} \otimes \sigma_{\alpha, \beta}(\Sigma)\right]^{\frac{1}{2}}\; \|f\|_{L^2(\R,  A_{\alpha, \beta})}~\|g\|_{L^2(\R,  A_{\alpha, \beta})}.$$
	Thus,  using   Heisenberg's inequality (\ref{eq36}), we get 
	\begin{align*}
	\left\|\W_{g}^{(\alpha, \beta)}(f)\right\|_{L^{2}\left(\Sigma,  A_{\alpha, \beta} \right)} \leq \frac{\left[A_{\alpha, \beta} \otimes \sigma_{\alpha, \beta}(\Sigma)\right]^{\frac{1}{2}}}{\varepsilon_{0}^{s}\left[1-A_{\alpha, \beta} \otimes \sigma_{\alpha, \beta}(B_{\varepsilon_{0}})\right]^{\frac{1}{2}}}\left\||(x, \xi)|^{s}\W_{g}^{(\alpha, \beta)}(f)\right\|_{L^2(\R^2,  A_{\alpha, \beta}\otimes \sigma_{\alpha, \beta})}.
	\end{align*}
	This completes the proof of the theorem with $c(s, \alpha, \beta)=\varepsilon_{0}^{-s}\left[1-A_{\alpha, \beta} \otimes \sigma_{\alpha, \beta}(B_{\varepsilon_{0}})\right]^{-\frac{1}{2}}.$
\end{proof}
The following result shows  that, Theorem \ref{eq37} gives a general form of Heisenberg-type  inequality    with a different constant. 
\begin{corollary}
	Let $s>0 .$ Then there exists a constant $c_{s, \alpha, \beta}>0$ such that, for all $f, g \in$ $L^2(\R,  A_{\alpha, \beta})$, we have 
	\begin{align}\label{eq38}
	\left\||(x, \xi)|^{s}\W_{g}^{(\alpha, \beta)}(f)\right\|_{L^2(\R^2,  A_{\alpha, \beta}\otimes \sigma_{\alpha, \beta})}\geq  c_{s, \alpha, \beta} ~\|f\|_{L^2(\R,  A_{\alpha, \beta})}\;\|g\|_{L^2(\R,  A_{\alpha, \beta})}.
	\end{align}
	In particular,
	\begin{align*}
	&\left\|x^s \W_{g}^{(\alpha, \beta)}(f)\right\|_{L^2(\R^2,  A_{\alpha, \beta}\otimes \sigma_{\alpha, \beta})}^2  +\left\|\xi ^s \W_{g}^{(\alpha, \beta)}(f)\right\|_{L^2(\R^2,  A_{\alpha, \beta}\otimes \sigma_{\alpha, \beta})}^2 \\
	&\geq  \frac{c_{s, \alpha, \beta }^2}{2^s}\; \|f\|_{L^2(\R,  A_{\alpha, \beta})}^{2}\;\|g\|_{L^2(\R,  A_{\alpha, \beta})}^{2}.\end{align*}
\end{corollary} 
\begin{proof}
	Let $r>0$  and $B_{r}=\{(x, \xi) \in \R^2 : ~|(x, \xi)|<r \}$ be the ball with  centre at origin  and radius $r$ in 
	$\R^2$.   Using   Plancherel's formula (\ref{eq17}) and Theorem  \ref{eq37}, we  obtain
	$$
	\begin{aligned}
	\Vert f \Vert_{L^2(\R,  A_{\alpha, \beta})}^2 \; \Vert g \Vert_{L^2(\R,  A_{\alpha, \beta})}^2
	&=	\left\Vert \W^{(\alpha, \beta)}_g(f) \right\Vert_{L^2(\R^2,  A_{\alpha, \beta}\otimes \sigma_{\alpha, \beta})}^2 \\
	&=	 	 \left \|  \chi_{B_r} \W^{(\alpha, \beta)}_g(f) \right\|_{L^2(\R^2,  A_{\alpha, \beta}\otimes \sigma_{\alpha, \beta})}^2+\left \|   \chi_{B_r^{c}}  \W^{(\alpha, \beta)}_g(f) \right\|_{L^2(\R^2,  A_{\alpha, \beta}\otimes \sigma_{\alpha, \beta})}^2  \\
	& \leq c(s, \alpha, \beta)^2\left[A_{\alpha, \beta} \otimes \sigma_{\alpha, \beta}(B_r)\right]\left\||(x, \xi)|^{s}\W_{g}^{(\alpha, \beta)}(f)\right\|_{L^2(\R^2,  A_{\alpha, \beta}\otimes \sigma_{\alpha, \beta})}^2\\
	&\qquad + r^{-2s}\left\||(x, \xi)|^{s}\W_{g}^{(\alpha, \beta)}(f)\right\|_{L^2(\R^2,  A_{\alpha, \beta}\otimes \sigma_{\alpha, \beta})}^2.
	\end{aligned}
	$$
	We get  the inequality  (\ref{eq38}) by minimizing the right-hand side of the above  inequality over $r>0$. Now, proceeding  similarly  as in the proof of Theorem 
	\ref{eq39}, we obtain	
	\begin{align*}
	&2^s\left\|x^s \W_{g}^{(\alpha, \beta)}(f)\right\|_{L^2(\R^2,  A_{\alpha, \beta}\otimes \sigma_{\alpha, \beta})}^2  +2^s\left\|\xi ^s \W_{g}^{(\alpha, \beta)}(f)\right\|_{L^2(\R^2,  A_{\alpha, \beta}\otimes \sigma_{\alpha, \beta})}^2 \\
	&\geq  c_{s, \alpha, \beta }^2\; \|f\|_{L^2(\R,  A_{\alpha, \beta})}^{2}~\|g\|_{L^2(\R,  A_{\alpha, \beta})}^{2},\end{align*} and this completes the proof.
\end{proof}

\subsection{Heisenberg-type uncertainty inequality via the $k$-entropy}
In this subsection,  we study the localization of the $k$-entropy of the windowed Opdam--Cherednik transform over the space $\mathbb{R}^{2 }$. Before going to prove the main result, we first need the following definition. 
\begin{definition}\quad
	\begin{enumerate} 
			\item A probability density function $\rho$ on $\mathbb{R}^{2 }$ is a non-negative measurable function on $\mathbb{R}^{2 }$ satisfying
		$$
		\iint_{\mathbb{R}^{2 }} \rho(x, \xi)~d(A_{\alpha, \beta} \otimes \sigma_{\alpha, \beta})(x, \xi)  =1.
		$$
		\item 	Let  $\rho$  be a probability density function  on $\mathbb{R}^{2 }$. Then the $k$-entropy of  $\rho$   is defined by
		$$
		E_{k}(\rho):=-\iint_{\mathbb{R}^{2 }} \ln (\rho(x, \xi))~ \rho(x, \xi)~d(A_{\alpha, \beta} \otimes \sigma_{\alpha, \beta})(x, \xi),
		$$
		whenever the integral on the right-hand side is well defined. 
	\end{enumerate}
\end{definition}
In the following, we prove the   main result of this subsection.
\begin{theorem}
	Let $g \in L^2(\R,  A_{\alpha, \beta})$ be a non-zero window function and
	$f \in L^2(\R,  A_{\alpha, \beta})$   such that $f\neq 0$.  Then, we have 
	\begin{align}\label{eq42}
	E_{k}(|\W_{g}^{(\alpha, \beta)}(f)|^{2}) \geq -2 \ln \left(\|f\|_{L^2(\R,  A_{\alpha, \beta})} \;\|g\|_{L^2(\R,  A_{\alpha, \beta})}\right)  \|f\|_{L^2(\R,  A_{\alpha, \beta})}^{2}\;\|g\|_{L^2(\R,  A_{\alpha, \beta})}^{2}.
	\end{align}
\end{theorem}
\begin{proof}
	First, we assume that $ \|f\|_{L^2(\R,  A_{\alpha, \beta})}=\|g\|_{L^2(\R,  A_{\alpha, \beta})}=1$. For any $(x, \xi )\in \mathbb{R}^2$, using   the relation (\ref{eq21}), we get 
	$$
	|\W_{g}^{(\alpha, \beta)}(f)(x, \xi)| \leq \; \|f\|_{L^2(\R,  A_{\alpha, \beta})}\;\|g\|_{L^2(\R,  A_{\alpha, \beta})}=1.
	$$
	Consequently,  $\ln (|\W_{g}^{(\alpha, \beta)}(f)|)\leq 0$ and   therefore  $E_{k} \big(|\W_{g}^{(\alpha, \beta)}(f)|\big) \geq 0$. The desired inequality (\ref{eq42}) holds  trivially  if the entropy $E_{k} \big(|\W_{g}^{(\alpha, \beta)}(f)|\big)$ is infinite. Now, suppose  that the entropy $E_{k} \big(|\W_{g}^{(\alpha, \beta)}(f)|\big)$ is finite. Let $ f$ and $g  $ be two non-zero functions in $L^2(\R,A_{\alpha, \beta})$. We define 
	$$
	\phi=\frac{f}{\|f\|_{L^2(\R,A_{\alpha, \beta})}} ~\text { and }~ \psi=\frac{g}{\|g\|_{L^2(\R,A_{\alpha, \beta})}}.
	$$
	Then $\phi , ~ \psi \in  L^2(\R,A_{\alpha, \beta})$ with    $\|\phi\|_{L^2(\R,A_{\alpha, \beta})}=\|\psi\|_{L^2(\R,A_{\alpha, \beta})}=1$ and consequently 
	\begin{align}\label{eq43}
	E_{k} \big(|\W_{\psi}^{(\alpha, \beta)}(\phi)|\big) \geq 0.
	\end{align} 
	Since  $\W_{\psi}^{(\alpha, \beta)}(\phi)=\frac{1}{\|f\|_{L^2(\R,A_{\alpha, \beta})}\|g\|_{L^2(\R,A_{\alpha, \beta})} } \W_{g}^{(\alpha, \beta)}(f)$, we have 
	\begin{align*}
	E_{k} \big(|\W_{\psi}^{(\alpha, \beta)}(\phi)|^{2}\big)=& -\iint_{\mathbb{R}^{2 }} \ln (|\W_{\psi}^{(\alpha, \beta)}(\phi)(x, \xi)|^{2})~ |\W_{\psi}^{(\alpha, \beta)}(\phi)(x, \xi)|^{2} ~d(A_{\alpha, \beta} \otimes \sigma_{\alpha, \beta})(x, \xi) \\
	& =\frac{1}{ \|f\|_{L^2(\R,A_{\alpha, \beta})}^2\|g\|_{L^2(\R,A_{\alpha, \beta})}^2}  E_{k} ( |\W_{g}^{(\alpha, \beta)}(f) | ^{2} ) \\& \qquad +2 \ln \big(\|f\|_{L^2(\R,A_{\alpha, \beta}) }\;\|g\|_{L^2(\R, A_{\alpha, \beta}) }\big).
	\end{align*}
	Finally,  from (\ref{eq43}),  we obtain  
	$$
	E_{k} ( |\W_{g}^{(\alpha, \beta)}(f) |^{2} ) \geq -2 \ln \left( \|f\|_{L^2(\R,A_{\alpha, \beta})}\;\|g\|_{L^2(\R,A_{\alpha, \beta})}\right) \|f\|_{L^2(\R,A_{\alpha, \beta})}^{2}\;\|g\|_{L^2(\R,A_{\alpha, \beta})}^{2}.
	$$
	This completes the proof of the theorem.
\end{proof}

\subsection{Application in signal processing}

Here, we present an application of these uncertainty principles in compressive sensing. Mainly, we show that uncertainty principles can be used for the separation of signals. The signal separation problem is an extremely ill-defined signal processing problem, which is also important in many engineering problems. It consists in splitting a signal $f$ into a sum of components $f_k$ of different nature: $f = f_1 + f_2 + \cdots +f_n$. Since this notion of different nature often makes sense in applied domains, it is generally extremely difficult to formalize mathematically. Sparsity offers a convenient framework for approaching such a notion. More precisely, signals of different natures can be sparsely represented in different waveform systems. Given a union of several frames $\mathcal U^{(1)}, \; \mathcal U^{(2)}, \cdots, \; \mathcal U^{(n)}$ in a Hilbert space $\mathbb H$, the separation problem can be given various formulations, among which the analysis and synthesis formulations (see \cite{ric14}). 

Let $U^{(k)}$ denotes the analysis operator of frame $k$. In the case of $n$ frames, applying these uncertainty principles and using the similar method as in \cite{ric14}, it can be proven that if one frame gives a splitting $f= f_1 + f_2 + \cdots + f_n$, obtained via any algorithm, if $\left\|U^{(1)} f_1\right\|+\left\|U^{(2)} f_2\right\| + \cdots + \left\|U^{(n)} f_n\right\|$ is small enough, then this splitting is necessarily optimal. More precisely, we have the following result. Let $\mathcal{U}^{(1)}, \; \mathcal{U}^{(2)}, \cdots, \; \mathcal U^{(n)}$ denote $n$ frames in $\mathbb{H}$. For any $f \in \mathbb{H}$, let $f= f_1 + f_2+ \cdots +f_n$ denote a splitting such that
$$ \left\|U^{(1)} f_1\right\|+\left\|U^{(2)} f_2\right\|+\cdots +\left\|U^{(n)} f_n\right\| < \frac{1}{\mu_{\star}}, $$
where $\mu_{\star}$ is the generalized coherence function (see \cite{ric14}). Then, using these uncertainty principles for the windowed Opdam--Cherednik transform, we obtain that this splitting minimizes $\left\|U^{(1)} f_1\right\|+\left\|U^{(2)} f_2\right\|+ \cdots + \left\|U^{(n)} f_n\right\|$. Similarly, using these uncertainty principles one can study sparsity-based algorithms for window optimization in time-frequency analysis. Many other applications can be given using uncertainty principles for the windowed Opdam--Cherednik transform.

\section*{Acknowledgments}
The first author gratefully acknowledges the support provided by IIT Guwahati, Government of India. The second author is deeply indebted to Prof. Nir Lev for several fruitful discussions and generous comments. The authors wish to thank the anonymous referees for their helpful comments and suggestions that helped to improve the quality of the paper.

\section*{Conflict of interest} 
No potential conflict of interest was reported by the authors.

\section*{ORCID}

{\it Shyam Swarup Mondal}  https://orcid.org/0000-0002-9778-0757 

{\it Anirudha Poria}  https://orcid.org/0000-0002-0224-3642

\end{document}